\documentclass[11pt,reqno]{amsart}
\usepackage{amsfonts,amssymb,amsmath}
\usepackage{enumitem}

\usepackage{xcolor}

\usepackage{multirow}
\usepackage{float}

\definecolor{fgreen}{RGB}{44,144, 14}

\renewenvironment{proof}{{\bfseries Proof.}}{\qed}
\numberwithin{equation}{section} 
\newtheorem{theorem}{Theorem}[section] 
\newtheorem{proposition}[theorem]{Proposition} 
 
\newtheorem{lemma}[theorem]{Lemma} 

\theoremstyle{definition}
\newtheorem{definition}[theorem]{Definition} 
\newtheorem{remark}[theorem]{Remark} 
\newtheorem{example}[theorem]{Example}

\usepackage[colorlinks=true,citecolor=cyan, urlcolor=blue,linkcolor=blue]{hyperref}

\def\R{\mathbb R}
\def\s{\mathbb S}
\def\C{\mathbb C}

\def\H{\mathbb H}

\def\ib{\mathbf {i}}
\def\jb{\mathbf {j}}
\def\kb{\mathbf {k}}

\def\D{\mathbb D}

\def\R{\mathbb R}

\def\g{\mathcal G}

\newcommand{\GL}{\mathrm{GL}}

\def\R{\mathbb {R}}
\def\C{\mathbb {C}}

\def\H{\mathbb {H}}

\def\ib{\mathbf {i}}

\def\jb{\mathbf {j}}

\def\g{\mathfrak {g}}

\def\s{\mathfrak {s}}

\def\o{\mathfrak {o}}
\def\l{\mathfrak {l}}

\def\lto{\longrightarrow}
\def\GL{\rm GL}

\newcommand{\thmref}[1]{Theorem~\ref{#1}}
\newcommand{\lemref}[1]{Lemma~\ref{#1}}

\newcommand{\propref}[1]{Proposition~\ref{#1}}

\begin{document}

\title[Reversibility in $ {\rm GL}(n,\D) $ ]{Reversibility and Real Adjoint Orbits of Linear Maps} 
 \author[K. Gongopadhyay,  T. Lohan   and  C. Maity]{Krishnendu Gongopadhyay, Tejbir Lohan  and  Chandan Maity}

\address{Indian Institute of Science Education and Research (IISER) Mohali,
 Knowledge City,  Sector 81, S.A.S. Nagar 140306, Punjab, India}
\email{krishnendug@gmail.com, krishnendu@iisermohali.ac.in}

 \address{Indian Institute of Science Education and Research (IISER) Mohali,
 Knowledge City,  Sector 81, S.A.S. Nagar 140306, Punjab, India}
\email{tejbirlohan70@gmail.com,  ph18028@iisermohali.ac.in}

\address{Indian Institute of Science Education and Research (IISER) Mohali,
 Knowledge City,  Sector 81, S.A.S. Nagar 140306, Punjab, India}
 \email{maity.chandan1@gmail.com , cmaity@iisermohali.ac.in }

\makeatletter
\@namedef{subjclassname@2020}{\textup{2020} Mathematics Subject Classification}
\makeatother

\maketitle 

 {\centering\footnotesize \it Dedicated to the 80th birthday of Norbert A'Campo.\par}

\begin{abstract}
 We extend classical results on the classification of reversible  elements of the group $\mathrm{GL}(n, \C)$  (and  $\mathrm{GL}(n, \R)$)  to 
$\mathrm{GL}(n, \H)$ using an infinitesimal version of the classical reversibility, namely adjoint reality in the Lie algebra set-up.  
We also provide a new proof of such a classification for the general linear groups over $\R$ and $\C$. Further, we classify the real adjoint orbits in the Lie algebra $\mathfrak{gl}(n, \D)$ for $ \D=\R, \C$ or $\H $.
 \end{abstract}


\section{Introduction} \label{sec-intro-1}

Reversing and time-reversing symmetries are important classes of symmetries that appear in the natural science. Especially they arise in many physically motivated dynamical systems. There is an extensive literature discussing such symmetries in different areas of physics and dynamical systems, see e.g. \cite{l}, \cite{lr}. In a mathematical terminology,  the motions of a dynamical system may be associated with a group $G$, and such symmetries correspond to the ``reversible" and ``strongly reversible"  elements in $G$.   An element $g$ in a group $G$ is called  \emph{reversible}  if $g$ and $g^{-1}$ are conjugate in $G$ , that is, if there exists $h \in G$ such that $hgh^{-1} = g^{-1}$. An element $g$ in a group $G$ is  called \emph{strongly reversible}  if $g$ is conjugate to $g^{-1}$ by an involution (i.e., an element of order at most two) in $G$.    Equivalently,  strongly reversible elements are products of two involutions.   Some authors have called strongly reversible elements  ``bireflectional". 

\medskip Reversible maps have appeared from different perspectives in the literature, e.g. \cite{Ar}, \cite{De}, \cite{Se}, \cite{FS}.   From a group-theoretical perspective, a classical theorem of Frobenius and Schur asserts that the number of real-valued complex irreducible characters of a finite simple group $G$ is equal to the number of reversible conjugacy classes of $G$.  With this motivation,  many mathematicians have used the terminology  ``real" and  ``strongly real" elements.   A strongly reversible element is reversible, but the converse is not true in general.   It is a problem of potential interest to classify reversible and strongly reversible elements in different groups of interest.  In the theory of finite groups, the classification of such elements is relatively well-understood in the literature.   However, a complete classification of reversible classes is not available other than for a  few families of infinite groups.  Some of the infinite groups where it has been classified include compact Lie groups, real rank one classical groups, and isometry groups of Hermitian spaces, see e.g. \cite{FS, BG, GL}.

\medskip The idea of reversibility is apparent in the work of A'Campo \cite{ac} as well. A'Campo investigated the monodromy of real isolated singularities using the fact that the complex conjugation on complex space permutes the level sets of a real polynomial function and induces involutions on level sets corresponding to real values. In particular,  it was proved that the geometric monodromy is the composition of the involution induced by complex conjugation and another involution. In other words, the geometric monodromies are strongly reversible. A'Campo called the corresponding singularity  \emph{strongly invertible}.  Further,  it follows that any two geometric monodromies are ``linked'' by the involution that comes from complex conjugation.

\medskip Recently,  the concept of reversibility has been extended to semisimple Lie algebras using the adjoint representations of Lie groups in \cite{GM}. The infinitesimal notion that has been introduced for the Lie algebras is called \emph{adjoint reality}.   Understanding the adjoint orbits of a semisimple Lie group is an active research theme, see the survey \cite{CoM}. However, the exploration of adjoint reality properties has been the object of attention only very recently, cf. \cite{GM}, \cite{GLM}.   As an application of adjoint reality for nilpotent orbits in simple Lie algebras,   the reversible and strongly reversible unipotent elements in the classical simple Lie groups have been completely classified,  see  \cite{GM}.   

\medskip  Let $\D=\R, \C$ or $\H$. 
In this chapter, we revisit the classification of reversible elements in the general linear group ${\rm GL}(n,\D)$. The classification of reversible elements in ${\rm GL}(n,\D)$ and their equivalence with the strongly reversible elements are well known in the literature for $\D=\R$ or $\C$, cf. \cite{Wo}, \cite{FS}. Extending these results over the quaternions is not straightforward due to the non-commutativity of $ \H$.    We shall overcome such difficulties by approaching this problem using adjoint reality in the Lie algebra $\g \l(n,\D)$. We recall the notion of adjoint reality below.

\medskip Consider the adjoint action of the general linear group $G:= {\rm GL}(n, \D)$ on its Lie algebra $\g:=\mathfrak{gl}(n,\D)$. Recall that $\mathfrak{gl}(n,\D) \simeq \mathrm{M}(n,\D)$, the algebra of  $n \times n$ matrices over $\D$.  In this case,  the adjoint action is given by the conjugation. Consider the conjugacy action of ${\rm GL}(n, \D)$ on $\g\l(n,\D)$: ${\rm Ad}(g)X:=gXg^{-1}$.    An element $X\in \g$ is called {\it $ {\rm Ad}_G$-real}  if $-X =gXg^{-1} $ for some $g\in G$.  An  $ {\rm Ad}_G$-real element $X\in \g$ is called {\it strongly $ {\rm Ad}_G$-real } if $-X = \tau X \tau^{-1} $ for some involution (i.e.,  element of order at most two)  $\tau\in G$; see \cite[Definition 1.1]{GM}. Observe that if $X\in \g$ is  $ {\rm Ad}_G$-real, then $\exp (X)$ is reversible in $G$, but the converse may not be true. 

\medskip We classify the adjoint real elements in $\mathfrak{gl}(n,\D)$ and investigate their equivalence with the strongly adjoint real elements in $\mathfrak{gl}(n,\D)$; see Theorems \ref{thm-adjoint-real-gl(n,D)}, \ref{thm-str-reality-gl-R-C}, \ref{thm-strong-real-gl(n,H)}.  Using these ideas and applying some of these results,  we shall prove that some particular types of Jordan forms in ${\rm GL }(n,\D)$ are strongly reversible. This will be used to classify the reversible and strongly reversible elements in ${\rm GL}(n, \D)$.

\medskip This approach not only reduces the complexity of the computations but also gives a better understanding of reversibility in the group $ {\rm GL}(n, \D)$. This also provides a uniform treatment over $ \D=\R, \C$ or $ \H$.  We show by a counterexample in Example \ref{ex-real-not str real-gl(1,H)} that,  unlike the field case,  the notion of reversible and strongly reversible elements are not equivalent in ${\rm GL}(n, \H)$ in general.    We give a sufficient criterion for equivalence of the two notions in ${\rm GL}(n,\H)$,  see Theorem \ref{thm-strong-rev-GL(n,H)}.

\medskip The chapter is organized as follows. In \S \ref{sec-prel-2}, we fix some notation and recall some background related to Jordan canonical forms. In \S \ref{sec-reality-Jordan-3}, reversibility and adjoint reality of certain Jordan forms are described. In \S \ref{sec-gl(n,D)-4}, we deal with the adjoint reality in the Lie algebra $ \g\l(n,\D)$. We revisit the reversibility problem of $ {\rm GL}(n,\D)$ in \S \ref{sec-GL(n,D)-5} and provide a much simpler proof of earlier obtained results.

\subsection*{Acknowledgement} It has been a highly rewarding experience to know Norbert A'Campo. Interaction with him has always been rich in mathematical and non-mathematical ideas!  It is a pleasure to dedicate this small contribution to the volume in honor of Norbert.  We wish Norbert a very good health and a happy life ahead.  

\medskip
Gongopadhyay is partially supported by the SERB core research grant 
CRG/2022/003680. 
Lohan acknowledges support from the CSIR SRF grant, File No.\,: 09/947(0113)/2019-EMR-I. 
Maity is supported by an NBHM PDF during this work.

\section{Preliminaries} \label{sec-prel-2}

In this section, we will recall some necessary background.
Recall that $\H:=\R\, +\,\R\ib +\,\R\jb+\,\R\kb$ denotes  the division algebra  of Hamilton’s quaternions.  We consider $\H^n$ as a right $\H$-module. We refer to \cite{rodman} for a nice exposition on quaternion linear algebra.

\begin{definition}\label{def-eigen-M(n,H)}
	Let $A \in  \mathrm{M}(n,\H)$. A non-zero vector $v \in \H^n $ is said to be a  right eigenvector of $A$ corresponding to a  right eigenvalue  $\lambda \in \H $ if the equality $ A\lambda = v\lambda $ holds.
\end{definition}

Eigenvalues of $A\in  \mathrm{M}(n,\H)$ occur in similarity classes, i.e., if $v$ is an eigenvector corresponding to $\lambda$,  then $v \mu \in v \mathbb H$ is an eigenvector corresponding to  $\mu^{-1} \lambda \mu$. Each similarity class of eigenvalues contains a unique complex number with a non-negative imaginary part. Here,  instead of similarity classes of eigenvalues, we will consider the \textit{unique complex representative} with a non-negative imaginary part.

Let $ \psi \colon \C\lto {\rm M}(2,\R) $ be the embedding given by $ \psi (z): = \begin{pmatrix} {\rm Re}(z)   & {\rm Im}(z)  \\
	- {\rm Im}(z)  & {\rm Re}(z)   \\ 
\end{pmatrix}$. 
This induces the  embedding 
$	\Psi \colon  \mathrm{M}(n,\C)  \longrightarrow  \mathrm{M}(2n,\R) \hbox{ defined as} $
	\begin{equation}\label{eq-embedding-psi}  \quad \Psi((z_{i,j})_{n \times n} ) \,:=    \, \big(\psi { (z_{i,j}}) \big)_{2n \times 2n}    \,.
\end{equation}
It follows from the definition of  the exponential map $ \exp:  \mathrm{M}(n,\C)  \longrightarrow  \mathrm{GL}(n,\C) $ that 
\begin{equation}\label{eq1-rel-embedding-Theta}
	\Psi(\exp (X))  \,=\, \exp(\Psi(X))   \quad \forall \   X \in {\rm M}(n, \C). 
\end{equation}

\begin{definition}[cf.~{\cite[p.  94]{rodman}}] \label{def-jordan}
	A {\it Jordan block} $\mathrm{J}(\lambda,m)$ is an $m \times m$ matrix with $ \lambda \in \D$ on the diagonal entries,  $1$ on all of the super-diagonal entries and zero elsewhere. We will refer to a block diagonal matrix where each block is a Jordan block as a \textit{Jordan form}. 
\end{definition}

We also consider the following block matrix as a Jordan form over $ \R $,  cf.~{\cite[p.  364]{GLR}}, which corresponds to the case when the eigenvalues of a matrix over $\R$ belong to $\C \setminus \R$. Recall the embedding $ \Psi $ as in \eqref{eq-embedding-psi}. Let  $K :=\Psi (\mu + \ib \nu) =   \begin{pmatrix}
	\mu   & \nu   \\
	-\nu &   \mu   \\
\end{pmatrix}  \in \mathrm{M}(2,\R)$,  where $\mu$, $\nu$ are real numbers with $\nu >0$.  Then define 
\begin{equation}\label{equ-real-Jordan block}
	\mathrm{J}_{\R}(\mu \pm \ib \nu, \,  2n )  := \Psi (\mathrm{J}(\mu + \ib \nu, \,  n )   ) \,=\,  \begin{pmatrix}
		K   & \mathrm{I}_2 &  & &    &    &  \\
		&  K   & \mathrm{I}_2 &  &   &    &  \\
		& &     &  \ddots &   &   \\
		&& & &    &   K   & \mathrm{I}_2 \\
		&  &&  &   &   &   K   \\
	\end{pmatrix} \, \in \mathrm{M}(2n,\R),
\end{equation}
where $ \mathrm{I}_2 $ denotes the $ 2\times 2$ identity matrix; see  {\cite[Theorem 15.1.1]{rodman}},  {\cite[Chapter 12]{GLR}}.  Further,  we also define 
\begin{equation}\label{eq-relation-real-jordan-block}
	\mathrm{J}_{\R}(\mu \mp \ib \nu, \,  2n ) := \Psi (\mathrm{J}(\mu - \ib \nu, \,  n )   )   =  \sigma  \,  \Big(  \Psi (\mathrm{J}(\mu + \ib \nu, \,  n )   ) \Big) \, \sigma^{-1},  
\end{equation}
where $  \sigma =   \textnormal{diag} (1,  -1,  1,  -1, \dots , (-1)^{2n-1} \ )_{2n \times 2n }$.

\begin{remark}
We will follow the notation  $ \mathrm{J}_{\R}(\mu \pm \ib \nu, \,  n ) $ and $ \mathrm{J}_{\R}(\mu \mp \ib \nu, \,  n ) $ as defined in \eqref{equ-real-Jordan block} and \eqref{eq-relation-real-jordan-block} throughout this chapter. Note that  $ \{  \mathrm{J}_{\R}(\mu \pm \ib \nu, \,  n ) \} $ is a singleton  by the above  definition.
\end{remark}

Next we recall the Jordan form in $\mathrm{M}(n,\D)$,  see {\cite[Theorem 15.1.1,  Theorem 5.5.3]{rodman}}.

\begin{lemma}[{Jordan  form in $ \mathrm{M}(n,\D)$, cf.~	 \cite{rodman}}] \label{lem-Jordan-M(n,D)}
	For every $A \in  \mathrm{M}(n,\D)$,  there is an invertible matrix $S \in  \mathrm{GL}(n,\D)$ such that $SAS^{-1}$ has the following  form: 
	\begin{enumerate}
		\item For $\D= \R$,	$	SAS^{-1} =  \mathrm{J}(\lambda_1,  \,  m_1) \oplus  \cdots \oplus  \mathrm{J}(\lambda_k, \, m_k)  $
		\begin{equation}\label{equ-Jordan-M(n,R)}
			\bigoplus  \mathrm{J}_{\R}( \mu_1 \pm \ib \nu_1, \,  2 \ell_1 )\oplus \cdots \oplus \mathrm{J}_{\R}(\mu_q \pm \ib \nu_q,  \, 2 \ell_q ), 
		\end{equation}
		where $ \lambda_1,  \dots,  \lambda_k  $; $ \mu_1,  \dots,  \mu_q  $ ; $ \nu_1,  \dots,  \nu_q  $  are  (not necessarily distinct) real numbers and $ \nu_1,  \dots,  \nu_q  $ are positive.  
		\item For $\D= \C$,	
		\begin{equation} \label{equ-Jordan-M(n,C)}
			SAS^{-1} =  \mathrm{J}(\lambda_1,  \,  m_1) \oplus  \cdots \oplus  \mathrm{J}(\lambda_k, \, m_k),
		\end{equation}
		where  $ \lambda_1,  \dots,  \lambda_k $ are  (not necessarily distinct) complex numbers.
		\item For $\D= \H$,
		\begin{equation} \label{equ-Jordan-M(n,H)}
			SAS^{-1} =  \mathrm{J}(\lambda_1, m_1) \oplus  \cdots \oplus  \mathrm{J}(\lambda_k, m_k),
		\end{equation}
		where $ \lambda_1,  \dots,  \lambda_k $ are  (not necessarily distinct) complex numbers and have non-negative imaginary parts.  
	\end{enumerate}
	The forms  (\ref{equ-Jordan-M(n,R)}),   (\ref{equ-Jordan-M(n,C)}) and  (\ref{equ-Jordan-M(n,H)}) are uniquely determined by $A$ up to a permutation of Jordan blocks.
\end{lemma}

\section{Strong Reversibility of Jordan forms}\label{sec-reality-Jordan-3}
In this section,  we  investigate the adjoint reality in $\mathfrak{gl}(n,\D)$ and reversibility in $\mathrm{GL}(n,\D)$  for certain types of Jordan forms. 
\subsection{Strong adjoint reality of Jordan forms in  $ \mathfrak{gl}(n,\D)$ }\label{subsec-str-ad-reality-Jordan}

Here, we will consider some particular types of  Jordan forms in  $\mathfrak{gl}(n,\D)$ and show that they are strongly $ {\rm Ad}_{{\rm GL }(n,\D)} $-real by explicitly constructing a suitable reversing involution.

\begin{lemma}\label{lem-reverser-lie-nilpotent}
	
Let  $ \D = \R, \C$ or $\H$, and  $X:= \mathrm{J}(0,  \,  n) $ be the nilpotent element in $ \mathfrak{gl}(n,\D) $.  Then $X$ is strongly ${\rm Ad}_{{\rm GL }(n,\D)} $-real.
\end{lemma}

\begin{proof} 
Let  $g :=   \textnormal{diag} (1,  -1,  1,  -1, \dots , (-1)^{n-1} \ )_{n \times n }. $ 
Then $g$ is an involution in ${\rm GL }(n,\D)$ such that $ gXg^{-1} = -X$.  This completes the proof.
\end{proof}

\begin{lemma}\label{lem-reverser-lie-pair-C}
Let  $X:=   \mathrm{J}(\lambda, n) \,  \oplus  \, \mathrm{J}(-\lambda, n) $ be the Jordan form in  $ \mathfrak{gl}(2n,\D)  $,   where $ \lambda \in \D \setminus \{0\},  $ for $ \D = \R$ or $\C, $ and for $ \D= \H$,  $\lambda  \in \C \setminus \{0\} $  with non-negative imaginary part  such that the real part of $ \lambda \neq 0$.  Then $X$ is  strongly $ {\rm Ad}_{{\rm GL }(2n,\D)} $-real.
\end{lemma}

\begin{proof} 
Write $X= \begin{pmatrix}
		\mathrm{J}(\lambda, n) &    \\
		&   \mathrm{J}(-\lambda, n)   \\
	\end{pmatrix}$. 
Let  $\tau := \textnormal{diag} (1,  -1,  1,  -1, \dots , (-1)^{n-1} \ )_{n \times n }$.  Note that $  \mathrm{J}(\lambda, n) \, \tau=-\tau \,  \mathrm{J}(-\lambda, n)$.  Consider $g =  \begin{pmatrix}
		&  \tau  \\
		\tau &     \\
	\end{pmatrix} \in  \mathrm{GL}(2n,\D) $.  Then $g$ is an involution in ${\rm GL }(2n,\D)$ such that $ gXg^{-1} = -X$.  This completes the proof.
\end{proof}

\begin{lemma} \label{lem-reverser-lie-pair-H}
Let $X:= \mathrm{J}(\mu\ib,   \,  n) \, \oplus \,  \mathrm{J}(\mu \ib,  \,  n)$ be the Jordan form in  $  \mathfrak{gl}(2n,\H)$,  where $\mu \in \R, \mu >0$.  Then $X$ is strongly ${\rm Ad}_{{\rm GL }(2n,\H)} $-real.
\end{lemma}

\begin{proof} 
Let $ \tau := \mathrm{diag} (\jb,-\jb,\jb,-\jb,\dots,(-1)^{n-1}\jb)_{n \times n }$. 
Then $\tau^2 = - \mathrm{I}_n$  and $  \tau \,  \mathrm{J}(\mu\ib,   \,  n) \, \tau^{-1} $
$= - \mathrm{J}(\mu\ib,   \,  n) $.   Consider the involution $ g = \begin{pmatrix}
		& \tau \\
		-\tau& \\ 
	\end{pmatrix}$ in ${\rm GL }(2n,\H)$.  
Then $g$ is an involution in ${\rm GL }(2n,\H)$ such that $ gXg^{-1} = -X$.  This proves the lemma.
\end{proof}

Recall that the matrix $\mathrm{J}_{\R}( \mu \pm \ib \nu, \,  2n ) $ is defined   in  \eqref{equ-real-Jordan block}.

\begin{lemma} \label{lem-reverser-lie-over-R-1}
Let  $X:=\mathrm{J}_{\R}( 0 \pm \ib \nu, \,  2n ) $ be the Jordan block in $\mathfrak{gl}(2n,\R) $,   where $\nu \in  \R$ such that $\nu >0$.  Then $X$ is  strongly $ {\rm Ad}_{{\rm GL }(2n,\R)} $-real.
\end{lemma}

\begin{proof}
Let  $g :=   \textnormal{diag} (\mathrm{I}_{1,1},  -\mathrm{I}_{1,1},  \mathrm{I}_{1,1},  -\mathrm{I}_{1,1}, \dots , (-1)^{n-1} \mathrm{I}_{1,1} \ )_{2n \times 2n }, $  where $ \mathrm{I}_{1,1} := \begin{pmatrix}
		1 &  0 \\
		0 & -1 \\
	\end{pmatrix}$.
Then $g$ is an involution in ${\rm GL }(2n,\R)$ such that $ gXg^{-1} = -X$.  This completes the proof.
\end{proof}

We refer to  \eqref{eq-relation-real-jordan-block} for the notation $\mathrm{J}_{\R}( \mu \mp \ib \nu, \,  2 n )$.  

\begin{lemma}\label{lem-reverser-lie-pair-R}
Let  $X := \mathrm{J}_{\R}( \mu \pm \ib \nu, \,  2 n ) \,  \oplus \,  \mathrm{J}_{\R}( -\mu \mp \ib \nu, \,  2 n ) $ be the Jordan form $\mathfrak{gl}(4n,\R)$, where  $\mu,  \nu \in  \R$ such that $\mu \neq 0$ and $\nu >0$.  Then $X$ is  strongly $ {\rm Ad}_{{\rm GL }(4n,\R)} $-real.
\end{lemma}

\begin{proof} 
	Write $X= \begin{pmatrix}
		\mathrm{J}_{\R}( \mu \pm \ib \nu, \,  2 n ) &    \\
		&   \mathrm{J}_{\R}( -\mu \mp \ib \nu, \,  2 n )   \\
	\end{pmatrix}$.  
Consider  $g =  \begin{pmatrix}
		&  \tau  \\
		\tau  &     \\
	\end{pmatrix}$,  where $\tau := \textnormal{diag} (\mathrm{I}_{2},  -\mathrm{I}_{2},  \mathrm{I}_{2},  -\mathrm{I}_{2}, \dots , (-1)^{n-1} \mathrm{I}_{2} \ )_{2n \times 2n }$.  Observe that $   \mathrm{J}_{\R}( \mu \pm \ib \nu, \,  2 n ) \, \tau=-\tau \,   \mathrm{J}_{\R}( -\mu \mp \ib \nu, \,  2 n )  $.  Then $g$ is an involution in ${\rm GL }(4n,\R)$ such that $ gXg^{-1} = -X$.  This completes the proof.
\end{proof}

\subsection{Strong reversibility of Jordan forms in  ${\rm GL }(n,\D)$} \label{subsec-reverser-GL-nD}
We will apply the results obtained in \S \ref{subsec-str-ad-reality-Jordan} to provide a different proof of strong reversibility of some particular types of Jordan forms in ${\rm GL }(n,\D)$.
Such results are known for $ \D=\R, \C $. We shall extend these results over $ \H $. These results will be used in the \propref{Prop-str-rev-GL-R-C} and the \thmref{thm-strong-rev-GL(n,H)}.

\begin{lemma} \label{lem-rev-jodan-type-1overC}
	Let $ \D = \R, \C$ or $\H$, and $A :=\mathrm{J}(\mu,  \,  n)$ be the Jordan block in $ {\rm GL }(n,\D)$,  where $\mu \in \{ \pm 1 \}$.  Then $A$ is strongly reversible in  ${\rm GL }(n,\D)$.
\end{lemma}

\begin{proof}
	First, we will consider $\mu=1$. 
	Then ${\rm J}(1,n)\,=\,{\rm I}_n \,+\, {\rm J}(0,n)$. By setting $N:={\rm J}(0,n)$, we have $g Ng^{-1}=-N, \, g^2={\rm I}_n$, where $ g  $ is as in Lemma \ref{lem-reverser-lie-nilpotent}. This implies $g \,e^N g^{-1}=e^{-N}$.
	Since the Jordan form of  $e^N$ is ${\rm J}(1,n)$, $ {\rm J}(1,n)=\tau e^N \tau^{-1} $ for some $\tau\in {\rm GL}(n,\D)$. Now 
	\begin{align*}
		(\tau g \tau^{-1} ) {\rm J}(1,n) (\tau g^{-1} \tau^{-1}) =\tau g e^N g^{-1}\tau^{-1}=\tau e^{-N}\tau^{-1}=({\rm J}(1,n))^{-1} \,.
	\end{align*} 
	Hence, ${\rm J}(1,n)$ is strongly reversible. The case when $ \mu =-1 $ follows from the fact that $-{\rm J}(-1,n)  $ has Jordan form $ {\rm J}(1,n) $ and  ${\rm J}(1,n)$ is strongly reversible.
	This completes the proof.
\end{proof}

\begin{lemma} \label{lem-rev-jodan-type-2 over C}
	Let  $A:=\mathrm{J}(\lambda, n) \oplus \mathrm{J}(\lambda^{-1}, n)   $ be the Jordan form in $ {\GL}(2n,\D)$,   where $ \lambda \in \D \setminus \{\pm 1, 0\},  $ for $ \D = \R $ or $ \C, $ and for $ \D= \H$,  $\lambda  \in \C \setminus \{0\} $  with non-negative imaginary part  such that  $ |\lambda| \neq 1$.  Then $A$ is strongly reversible in  ${\rm GL }(2n,\D)$.
\end{lemma}

\begin{proof} 
	Let  $ \mu \in \C $ such that $ e^\mu =\lambda $.  Note that 
	$ \exp( {\rm J}(\mu, n)) \,=\, \exp(\mu {\rm I}_n )\cdot \exp({\rm J}(0, n)) \,  = \,  \lambda \,  \exp({\rm J}(0, n)).$ 	Let $ P \in  {\GL}(n,\D)$  so that $ \lambda  \, \exp({\rm J}(0, n)) \,=\, P  {\rm J}(\lambda, n) \, P^{-1} $. Thus
	\begin{equation}\label{eq-1}
		\exp( {\rm J}(\mu, n)) \,=\, P  {\rm J}(\lambda, n) \, P^{-1}.	\end{equation}
	Similarly, there exists   $ Q \in  {\GL}(n,\D) $ such that 
	\begin{equation}\label{eq-2} \exp( {\rm J}(-\mu, n)) \, = 	\, Q \,  {\rm J}(\lambda^{-1}, n) \, Q^{-1} .
	\end{equation}
	Now, 	let $ \sigma := {\rm diag} (1,-1, \dots , (-1)^{n-1}) $ be an involution in $
	{\GL}(n,\D) $ so that $ 
	-{\rm J}(-\mu, n) = \sigma \, {\rm J}(\mu, n) \, \sigma^{-1}$. This implies   
	\begin{equation}\label{eq-3} \sigma  \,  \exp({\rm J}(\mu, n))  \, \sigma^{-1}\,=\,   \Big( \exp( {\rm J}(-\mu, n)) \Big)^{-1}.
	\end{equation}
	Using \eqref{eq-1}, \eqref{eq-2} and \eqref{eq-3}, we have
	\begin{equation}\label{eq-Jordan-conj-inv}
		\sigma \, P \,  {\rm J}(\lambda, n) \, P^{-1} \,  \sigma^{-1} = \Big ( Q \,  {\rm J}(\lambda^{-1}, n) \,  Q^{-1} \Big)^{-1}.
	\end{equation}
	Consider the involution  $ g := \begin{pmatrix}
		P\\
		&Q
	\end{pmatrix}^{-1} \begin{pmatrix}
		& \sigma\\
		\sigma^{-1} &
	\end{pmatrix} \begin{pmatrix}
		P\\
		&Q
	\end{pmatrix}$ in $ {\GL}(2n,\D) $.  Then $gAg^{-1}$  $\,=\, A^{-1} $ if and only if  
	\begin{equation} \label{eq-12}
		(Q^{-1} \sigma \, P )\,  {\rm J}(\lambda, n) \, 	(Q^{-1} \sigma \, P )^{-1} = \Big ( {\rm J}(\lambda^{-1}, n) \Big)^{-1}\,. 
	\end{equation} 
	Now the proof follows from  Equation \eqref{eq-Jordan-conj-inv}.
\end{proof}

\begin{lemma} \label{lem-rev-jodan-unit-modulus- H}
	Let  $A := \mathrm{J}(\mu,  \,  n) \, \oplus \,  \mathrm{J}(\mu,  \,  n) $ be the Jordan block in $ {\rm GL }(2n,\H)$,  where $ \mu \in \C \setminus \{\pm 1\}$,  with non-negative imaginary part such that $|\mu|= 1$.  Then $A$ is strongly reversible in  ${\rm GL }(2n,\H)$.
\end{lemma}

\begin{proof} Recall that $ \jb \, Z = \overline{Z} \,  \jb$ for all $Z \in  {\rm GL }(n,\C)$, where $\overline{Z}$ is the matrix obtained by taking the conjugate of each entry of the complex matrix $Z$.
	We  can assume that $A =      \mathrm{J}(e^{ \ib \theta},   \,  n) \, \oplus \,      \mathrm{J}(e^{ \ib \theta},   \,  n)$,  where  $ \theta \in (0,\pi)$.
	Write $A = \begin{pmatrix}
		P & \\
		&   P \\ 
	\end{pmatrix}$, 
	where $P =     \mathrm{J}(e^{ \ib \theta},   \,  n) \in  {\rm GL }(n,\C)$.  To show $A$ is strongly reversible, it is sufficient to find a $g =    \begin{pmatrix}
		& X \\
		X^{-1} &  \\ 
	\end{pmatrix}
	\in {\rm GL }(2n,\H)$ such that $P^{-1}X = XP$, where $X \in {\rm GL }(n,\H)$.  Further,  If  $X = Y \jb $ for some $ Y  \in {\rm GL }(n,\C)$, then we require $P^{-1}Y = Y \overline{P} $, i.e.,    $  \Big(\mathrm{J}(e^{ \ib \theta}, n)\Big)^{-1} \, Y = Y \, \mathrm{J}(e^{ -\ib \theta}, n)$. 
	Using the construction as done in the proof of  \lemref{lem-rev-jodan-type-2 over C}, we can find such $Y$ in $ {\rm GL }(n,\C)$, see \eqref{eq-12}. This completes the proof.
\end{proof}

\begin{lemma} \label{lem-rev-jodan-type-1 over R} 
	Let  $A:={ \mathrm{J}_{\R}(\mu \pm \ib \nu, \,  2n ) } $ be the Jordan block in  $ {\rm GL }(2n,\R)$ as in  \eqref{equ-real-Jordan block}.  If $\mu^2 + \nu^2 =1$, then $A$ is strongly reversible in  ${\rm GL }(2n,\R)$.
\end{lemma}

\begin{proof} 
	Let $K:= \begin{pmatrix}
		\mu  & \nu   \\
		-\nu &   \mu   \\
	\end{pmatrix}$ where $\mu^2 + \nu^2 =1$. Then $K\in {\rm SO}(2)$. Let  $Y = \begin{pmatrix}
		0  & a  \\
		-a &   0 \\
	\end{pmatrix}\in \s\o(2)$ such that $ \exp (Y) \,=\,K $.   Note that for  $\mathrm{I}_{1,1} := {\rm diag}(1,-1)$,  we have $  \mathrm{I}_{1,1} \, Y \,  (\mathrm{I}_{1,1})^{-1} \,=\, -Y$.
	Consider $$ X:= { \mathrm{J}_{\R}(0 \pm \ib a, \,  2n ) }  
	= \begin{pmatrix}
		Y & {\rm I}_2\\
		& \ddots &\ddots\\
		&& 	Y & {\rm I}_2\\
		&& &	Y 
	\end{pmatrix}  \in  {\rm GL}(2n, \R).$$
	Let $ g:={\rm diag}(\mathrm{I}_{1,1}, -\mathrm{I}_{1,1}, \dots, (-1)^{n-1}\mathrm{I}_{1,1}) \in {\rm GL}(2n, \R) $. Then $ gXg^{-1}=-X$, and hence $ \exp(X) $ is strongly reversible. 
	To conclude the proof, it is sufficient to show that $ \exp(X) $ and $\mathrm{J}_{\R}(\mu \pm \ib \nu, \,  2n )$ are conjugate. To see this recall that $ \Psi ({\rm J}(\mu+\ib \nu, n)) \,=\, \mathrm{J}_{\R}(\mu \pm \ib \nu, \,  2n )$, and $ \Psi({\rm J}(\ib a, n) )\,=\, X $, where the map $\Psi  $ is as in \eqref{eq-embedding-psi}. Since $ \exp (Y) \,=\,K $,  we have $ e^{ia} = \mu+\ib \nu$. Then 
	\begin{equation}\label{eq-relation--embedding-Theta}
		\exp({\rm J}(\ib a, n) )  \,= \,(\mu + \ib \nu)\exp({\rm J}(0, n) ) \,=\, h \, {\rm J}(\mu+\ib \nu, \,n) \, h^{-1} \, ,
	\end{equation}
	for some $ h\in {\rm GL}(n,\C) $.
	By using Equations  \eqref{eq1-rel-embedding-Theta}  and \eqref{eq-relation--embedding-Theta},  we have 
	$$
	\exp (X)  = \Psi(\exp (\mathrm{J}(\ib a, \,  n ))  \,=  \, h \, \Big ( \Psi ( {\rm J}(\mu+\ib \nu, \,n))  \Big)  \, h^{-1}  =  \, h \, \mathrm{J}_{\R}(\mu \pm \ib \nu, \,  2n )   \, h^{-1}.$$
	Therefore,  $\mathrm{J}_{\R}(\mu \pm \ib \nu, \,  2n) $ is strongly reversible in ${\rm GL}(2n,\R) $.  This completes the proof.
\end{proof}

\begin{lemma} \label{lem-rev-jodan-type-2 over R}
	Let  
	$ A:= \mathrm{J}_{\R}(\mu \pm \ib \nu, \,  2n )\,  \oplus  \,  \mathrm{J}_{\R}\big(\frac{\mu}{\mu^2 + \nu^2 } \mp \ib \frac{\nu}{\mu^2 + \nu^2 }, \,  2n\big )$ 
	be the Jordan form  in $ {\rm GL }(4n,\R)$,  where $\mu^2 + \nu^2 \neq 1$.
	Then $A$ is strongly reversible in  ${\rm GL }(4n,\R)$.
\end{lemma}

\begin{proof}
	Let $ z\in \C $ such that  $ e^{z}\,=\, \mu+i\nu$ and  $ e^{-z} = \frac{\mu}{\mu^2 + \nu^2 } -\ib \frac{\nu}{\mu^2 + \nu^2 }$.  Recall that the embedding  $\Psi  $ is as in  \eqref{eq-embedding-psi}. Using \eqref{equ-real-Jordan block} and \eqref{eq-relation-real-jordan-block}, we have  $$ \Psi ({\rm J}(e^z, n)) = \mathrm{J}_{\R}(\mu \pm \ib \nu, \,  2n ) \, \hbox{  and  } \, \Psi({\rm J}(e^{-z} , n) )= \mathrm{J}_{\R}(\frac{\mu}{\mu^2 + \nu^2 } \mp \ib \frac{\nu}{\mu^2 + \nu^2 }, \,  2n )  .$$  Therefore, we can write $A$ as  
	\begin{equation}
		A = \Psi (P),  \hbox {where }  P =   \begin{pmatrix}
			{\rm J}(e^z, n) &   \\
			&  {\rm J}(e^{-z}, n)  \\
		\end{pmatrix} \in {\rm GL }(2n,\C).
	\end{equation}
	The proof now follows from  Lemma \ref{lem-rev-jodan-type-2 over C} and the fact that $\Psi$ is an embedding.
\end{proof}

\section{Adjoint reality in $\mathfrak{gl}(n,\mathbb{D} )$}\label{sec-gl(n,D)-4}

In this section, we investigate adjoint reality in $\mathfrak{gl}(n,\D)$. First we classify  $ {\rm Ad}_{{\rm GL }(n,\D)} $-real elements in the Lie algebra $\mathfrak{gl}(n,\D)$.

\begin{theorem} \label{thm-adjoint-real-gl(n,D)} 
Let $\mathbb{D} = \R, \C$ or $\H$.  An element $X \in \mathfrak{gl}(n,\D)$ with Jordan canonical form  as given in Lemma \ref{lem-Jordan-M(n,D)}  is $ {\rm Ad}_{{\rm GL }(n,\D)} $-real if and only if the following hold:
\begin{enumerate}

\medskip 

\item \label{part-adjoint-real-R}
For $ \D=\R $, the blocks can be partitioned into  pairs  $\{ \mathrm{J}_{\R}( \mu \pm \ib \nu, \,  2 t ),  \mathrm{J}_{\R}( -\mu \mp \ib \nu, \,  2 t ) \} $,  $ \{ \mathrm{J}(\lambda, s),\mathrm{J}(-\lambda, s)\} $ or, singletons  $ \{\mathrm{J}_{\R}(0 \pm \ib \nu, \,  2 \ell ) \}$, $\{\mathrm{J}(0, m  )\}$,  where $\lambda, \mu,\nu \in \R $ and $ \lambda, \mu \neq 0,  \nu > 0$.
		
\medskip 

\item  \label{part-adjoint-real-C} For $ \D=\C $, the blocks can be partitioned into   pairs $ \{ \mathrm{J}(\lambda, s),\mathrm{J}(-\lambda, s)\} $ or, singletons $\{\mathrm{J}(0, m  )\}$,  where $\lambda \in \C $ and $ \lambda \neq 0$.
		
\medskip 

\item \label{part-adjoint-real-H} For $ \D=\H $, the blocks can be partitioned into  pairs $ \{ \mathrm{J}(\lambda, s),\mathrm{J}(-\lambda, s)\} $ or, singletons $\{\mathrm{J}(\mu, m  )\}$,  where $\lambda, \mu \in \C $  with non-negative imaginary parts  such that real part of $ \lambda \neq 0$ and real part of $\mu =0$.
\end{enumerate}
\end{theorem}

\begin{proof} 
Consider the case $\D=\R$. Using Lemma \ref{lem-Jordan-M(n,D)},  $X$ is conjugate to $-X$ if and only if 
	$$ \{ \mathrm{J}(\lambda_1 ,  m_1  ),  \dots,  \mathrm{J}(\lambda_k ,  m_k  ),\mathrm{J}_{\R}( \mu_1 \pm \ib \nu_1, \,  2 \ell_1 ), \dots, \mathrm{J}_{\R}( \mu_q \pm \ib \nu_q, \,  2 \ell_q ) \}$$
	$$\,=\, \{ \mathrm{J}(-\lambda_1 ,  m_1  ),  \dots,  \mathrm{J}(-\lambda_k ,  m_k  ),\mathrm{J}_{\R}( -\mu_1 \mp \ib \nu_1, \,  2 \ell_1 ), \dots, \mathrm{J}_{\R}(- \mu_q \mp \ib \nu_q, \,  2 \ell_q ) \},$$
and the result follows immediately for the case $\D=\R$.  
Recall the fact that for a unique complex representative $\lambda$ of an eigenvalue class of  $X$,  $[\lambda]= [-\lambda]$ if and only if the real part of $\lambda =0$. Using the same line of argument as we used in the $\D=\R$ case, the result follows for the case $\D= \C$ or $\H$.  
\end{proof}

Recall that every strongly $ {\rm Ad}_{{\rm GL }(n,\D)} $-real element in $ \mathfrak{gl}(n,\D)$ is $ {\rm Ad}_{{\rm GL }(n,\D)} $-real.  In the next result, we will prove that the converse holds when $\D = \R$ or $ \C$.
\begin{theorem} \label{thm-str-reality-gl-R-C}
Let $\mathbb{D} = \R$ or $\C$.
An element   $A$ of  $ \mathfrak{gl}(n,\D)$ is $ {\rm Ad}_{{\rm GL }(n,\D)} $-real if and only if  it is strongly $ {\rm Ad}_{{\rm GL }(n,\D)} $-real.
\end{theorem}
\begin{proof}  Without loss of generality,  we can assume that $A$ is  in Jordan form as in \lemref{lem-Jordan-M(n,D)}.   Suppose  $A$ is  $ {\rm Ad}_{{\rm GL }(n,\D)} $-real.  Then Jordan blocks of $A$ can be partitioned as in \thmref{thm-adjoint-real-gl(n,D)}. In view of  \lemref{lem-reverser-lie-nilpotent}, \lemref{lem-reverser-lie-pair-C},   \lemref{lem-reverser-lie-over-R-1},   and \lemref{lem-reverser-lie-pair-R},  we can explicitly construct an involution  $ g$ in  $ {\rm GL }(n,\D)$   such that $gAg^{-1}=-A$.  Since the converse is trivial, this completes the proof.
\end{proof}

The following example shows that the above result does not hold for $\mathfrak{gl}(n,\H)$.

\begin{example}\label{ex-real-not str real-gl(1,H)}
Let $A := (\ib) \in  \mathfrak{gl}(1,\H)$.  Then $gAg^{-1}=- A$,  where $g= (\jb) \in \mathrm{GL}(1,\H)$.  So $A$ is $ {\rm Ad}_{{\rm GL }(1,\H)} $-real.  Suppose that $A$ is strongly $ {\rm Ad}_{{\rm GL }(1,\H)} $-real.  Let $g = (a) \in {\rm GL }(1,\H)$ be an involution such that  $gAg^{-1}= -A$.  Then we get $ a\ib = -\ib a$.  This implies $a = w\jb $ for some $ w \in \C,  \, w \neq 0$. Since $g$ is an involution,  $a^2 = (w \jb)^2 =1$, 
i.e,  $|w|^2 = -1$. This is a contradiction. Therefore,  $A$ is $ {\rm Ad}_{{\rm GL }(1,\H)} $-real  but not strongly $ {\rm Ad}_{{\rm GL }(1,\H)} $-real.  \qed
\end{example}

The next result gives a sufficient criterion for the $ {\rm Ad}_{{\rm GL }(n,\H)} $-real elements in $ \mathfrak{gl}(n,\H)$ to be strongly  $ {\rm Ad}_{{\rm GL }(n,\H)} $-real.   
\begin{theorem}\label{thm-strong-real-gl(n,H)}
Let $A  \in \mathfrak{gl}(n,\H)$ be an arbitrary $ {\rm Ad}_{{\rm GL }(n,\H)} $-real element.   Suppose that in the Jordan decomposition (\ref{equ-Jordan-M(n,H)}) of $A$,  every Jordan block corresponding to a non-zero eigenvalue class with purely imaginary  complex representative has even multiplicity. Then $A$ is strongly $ {\rm Ad}_{{\rm GL }(n,\H)} $-real. 
\end{theorem}

\begin{proof}
In view of   \thmref{thm-adjoint-real-gl(n,D)},  the proof  follows from   \lemref{lem-reverser-lie-nilpotent}, \lemref{lem-reverser-lie-pair-C} and \lemref{lem-reverser-lie-pair-H}.
\end{proof}

\section{Reversibility in  ${\rm GL }(n,\D)$}\label{sec-GL(n,D)-5}

As before,  $\mathbb{D} := \R, \C $ or $\H$.   Here,  we will consider reversibility in the Lie group $ {\rm GL}(n,\D) $. 
The classification of reversible elements in ${\rm GL }(n,\C)$ is given in \cite[Theorem 4.2]{ FS}.  We have included the case $ \D=\C $ here as it will be used in Proposition \ref{Prop-str-rev-GL-R-C}.

\begin{theorem}\label{thm-reversible-GL(n,D)} 
Let $\mathbb{D} := \R, \C $ or $\H$.  An element $A \in  {\rm GL }(n,\D)$ with Jordan form  as given in Lemma \ref{lem-Jordan-M(n,D)} is reversible  if and only if  the following hold: 
\begin{enumerate}
\medskip \item \label{part-reversible-GL(n,R)}
For $ \D=\R $, the blocks can be partitioned into pairs   $\{ \mathrm{J}_{\R} (\mu \pm \ib \nu, \,  2 t ),  \mathrm{J}_{\R}( \frac{\mu}{\mu^2 + \nu^2} \,  \mp \,  \ib \frac{\nu}{\mu^2 + \nu^2}, \,  2 t ) \} $, $ \{ \mathrm{J}(\lambda, s),\mathrm{J}(\lambda^{-1}, s)\} $ or, singletons  $ \{\mathrm{J}_{\R}(\alpha \pm \ib \beta, \,  2 \ell ) \}$, $\{\mathrm{J}(\gamma, m  )\}$,  where $\lambda, \mu,\nu \in \R $ such that $ \lambda, \gamma \neq 0,  \nu, \beta > 0$ and $ \lambda \neq \pm1$,  $\mu^2 + \nu^2 \neq  1, \gamma = \pm 1,\,   \alpha^2 + \beta^2 = 1$. 
		
\medskip \item   \label{part-reversible-GL(n,C)}
For $ \D=\C $, the blocks can be partitioned into   pairs $ \{ \mathrm{J}(\lambda, s),\mathrm{J}(\lambda^{-1}, s)\} $ or, singletons $\{\mathrm{J}(\mu, m  )\}$,  where $\lambda, \mu \in \C \setminus \{0\} $ and $ \lambda \neq \pm 1,  \mu = \pm 1  $.
	
\medskip \item   \label{part-reversible-GL(n,H)} For $ \D=\H $, the blocks can be partitioned into  pairs $ \{ \mathrm{J}(\lambda, s),\mathrm{J}(\lambda^{-1}, s)\} $ or, singletons $\{\mathrm{J}(\mu, m  )\}$,  where $\lambda, \mu \in \C \setminus \{0\}$  with non-negative imaginary parts  such that  $|\lambda| \neq 1,  |\mu| = 1$.  
\end{enumerate}
\end{theorem}

\begin{proof} 	
In the case of $\D =\H$,   for a unique complex representative $\lambda \in \C$ of an eigenvalue class of  $A$,  $[\lambda]= [\lambda^{-1}]$ if and only if $|\lambda|= 1$, i.e.,  $\lambda^{-1} = \overline{\lambda}$. 
Using  Lemma \ref{lem-Jordan-M(n,D)},  $A$ is conjugate to $A^{-1}$ if and only if  $A$ and $A^{-1}$ has same Jordan decomposition.  
Now the proof of $ (\ref{part-reversible-GL(n,R)}) $ and $ (\ref{part-reversible-GL(n,H)}) $ follows from the same line of arguments as done in 
	\thmref{thm-adjoint-real-gl(n,D)}.
For the proof of $ (\ref{part-reversible-GL(n,C)}) $,  we refer to \cite[Theorem 4.2]{ FS}. 
\end{proof}

\medskip 

In ${\rm GL }(n,\D) $,  every reversible element is strongly reversible for the case $\mathbb{D} = \R$  or $\C$.
\begin{proposition} [cf.~{\cite[ Theorems 4.7]{FS}}]  \label{Prop-str-rev-GL-R-C}
Let $A \in {\rm GL }(n,\D) $,  where $\mathbb{D} = \R$ or $ \C$.  Then $A$ is reversible  in   $ {\rm GL }(n,\D) $ if and only if $A$ is strongly reversible in   $ {\rm GL }(n,\D) $.
\end{proposition}

\begin{proof}
In view of  \thmref{thm-reversible-GL(n,D)}, the proof  follows from 
 \lemref{lem-rev-jodan-type-1overC},  \lemref{lem-rev-jodan-type-2 over C}, \lemref{lem-rev-jodan-type-1 over R},  and  \lemref{lem-rev-jodan-type-2 over R}. 
\end{proof}

The next example shows that \propref{Prop-str-rev-GL-R-C} does not hold in the case $\D =\H$.

\begin{example}\label{ex-real-not str real-gl(1,H)}
Let $A: = (\ib) \in {\rm GL }(1,\H)$. Then $A$ is reversible in ${\rm GL }(n,\H)$   but not strongly reversible.
\qed \end{example}

The next result gives a sufficient criterion for the reversible elements in $ {\rm GL }(n,\H)$ to be strongly reversible.
\begin{theorem}\label{thm-strong-rev-GL(n,H)}
Let $A \in {\rm GL }(n,\H)$ be an arbitrary reversible element.  Suppose that in the Jordan decomposition of $A$,  every Jordan block corresponding to non-real eigenvalue classes of unit modulus has even multiplicity.  Then $A$ is strongly reversible in ${\rm GL }(n,\H)$. 
\end{theorem}

\begin{proof}
In view of  \thmref{thm-reversible-GL(n,D)}, the proof  follows from     \lemref{lem-rev-jodan-type-1overC}, \lemref{lem-rev-jodan-type-2 over C} and \lemref{lem-rev-jodan-unit-modulus- H}.
\end{proof}

\end{document}